\documentclass[final,12pt,reqno]{amsart}

\usepackage{amssymb,amsmath,graphicx,amsfonts,euscript}
\usepackage{color}
\usepackage{graphicx}
\usepackage{color}
\usepackage{showkeys}
\usepackage{epstopdf}

\newtheorem{thm}{Theorem}[section]

\newtheorem{rem}[thm]{Remark}

\newtheorem{lemma}[thm]{Lemma}

\newcommand{\be}{\begin{eqnarray}}
\newcommand{\ee}{\end{eqnarray}}
\newcommand{\bes}{\begin{eqnarray*}}
\newcommand{\ees}{\end{eqnarray*}}

\voffset=-0.2in \numberwithin{equation}{section}
\begin{document}
\title[Magneto-micropolar equations]{The stabilizing effect of the microstructure on the 3D magneto-micropolar equations}

\author[Shang]{Haifeng Shang}

\address{School of Mathematics and Statistics, Northeastern University at Qinhuangdao, Qinhuangdao, 066004, China}

\email{hfshang@163.com}

\vskip .2in
\begin{abstract}

This paper focuses on the global stability of the 3D magneto-micropolar equations with partial viscosity in the torus $\mathbb T^3$. We first establish the global stability and exponential decay for the 3D magneto-micropolar equations with zero kinematic viscosity. If the micro-rotation effect is neglected, this system reduces to the 3D inviscid and resistive MHD equations which stability problem is still a challenging open problem. Secondly, we obtain the global stability and algebraic decay to the 3D magneto-micropolar equations with zero kinematic viscosity and zero magnetic diffusion on perturbations near a background magnetic field. This system becomes the 3D ideal MHD equations by ignoring the microstructure, and it is well-known that the weighted spaces must be introduced to show the global well-posedness of the ideal MHD equations. Our results indicate that the microstructure has the effect of enhancing dissipation and contributes to stabilize the fluid. To the best of our knowledge, these are the first results on the stabilizing effect of the microstructure on electrically conducting fluids.

\end{abstract}

\maketitle

\vskip .3in
\section{Introduction}
\vskip .1in

In this paper, we study the three dimensional (3D) incompressible magneto-micropolar equations in the periodic domain $\mathbb T^3$:
\begin{align}\label{MME}
\left\{ \begin{array}{lcl}
\partial_t u + u \cdot \nabla u - (\mu + \chi) \Delta u  = - \nabla p + 2\chi \nabla\times \omega + b\cdot \nabla b,\\
\partial_t \omega + u \cdot \nabla \omega + 4\chi \omega - \kappa \nabla\nabla\cdot \omega - \eta \Delta \omega = 2\chi \nabla\times u,\\
\partial_t b + u \cdot \nabla b - \nu \Delta b  =  b\cdot \nabla u,\\
\nabla \cdot u=0, \nabla \cdot b=0,\\
u(x, 0) = u_{0}(x), \omega(x, 0) = \omega_{0}(x), b(x, 0) = b_{0}(x),
\end{array} \right.
\end{align}
where $u$, $\omega$, $b$ and $p$ denote the fluid velocity, the micro-rotational velocity, the magnetic field and the pressure, respectively. The nonnegative constants $\mu$, $\chi$ and $\nu$ are the kinematic viscosity, the micro-rotation viscosity and the magnetic diffusivity, respectively, and $\kappa$ and $\eta$ are the angular viscosities.

The magneto-micropolar equations \eqref{MME} can describe fluids with microstructure (see \cite{AS74,Eringen66,Lukaszewicz99}) which cannot be represented by the classical fluids such as the Navier-Stokes equations and the magnetohydrodynamic (MHD) equations. It is worth noting that four new viscosities $\chi$, which size allows us to measure in a certain sense that the derivation of flows of micropolar fluids from that of the Navier-Stokes model or MHD model, are introduced in this model. If the micro-rotation viscosity $\chi = 0$, the conservation law of the linear momentum becomes independent of the presence of the microstructure and system \eqref{MME} reduces to the MHD equations.


When $\mu>0$, $\eta>0$ and $\nu>0$, the global existence of weak solutions in bounded domain and whole space has been shown in \cite{LS18,RB98,TWZ19}. It is well-known that the global regularity of the weak solutions and the global existence of the strong solution are classical for 2D case while the corresponding problems in 3D case are still challenging open problems.

When $\mu>0$, $\eta=0$ and $\nu>0$, the stress momentum is lost in rotation of the particles, the microstructure plays an important role as it usually increases the load capacity and stabilizes the flows (see \cite{Ferrari83,SD85}). Due to the lack of micro-rotational velocity dissipation, the global well-posedness and large time behavior issues of this kind of magneto-micropolar fluid are more difficult than \eqref{MME} with full dissipation. In the 2D space, Yamazaki \cite{Yamazaki15} has previously established the global regularity by fully exploiting the structure of the system and bounding the Lebesgue norm of the first derivatives of the solution (see also \cite{SW21}). By utilizing the special structure of this system and making full use of the velocity dissipation and the damping term in equation of $\omega$, Niu and Shang \cite{NS24} obtained the existence of small data global classical solutions and the lower and upper bounds of decay estimates of the global solutions for the 2D and 3D cases. For other exciting results, one can refer to \cite{BCFZ19,CM12,DLW17,DZ10,GMNPZ25,SL25,SS26,SZ17,WZ23,Xue11,YQ18}.


When $\mu=0$, $\eta>0$ and $\nu>0$, system \eqref{MME} reduces to the following system
\begin{align}\label{MMEZ}
\left\{ \begin{array}{lcl}
\partial_t u + u \cdot \nabla u - \chi \Delta u  = - \nabla p + 2\chi \nabla\times \omega + b\cdot \nabla b,\\
\partial_t \omega + u \cdot \nabla \omega + 4\chi \omega - \kappa \nabla\nabla\cdot \omega - \eta \Delta \omega = 2\chi \nabla\times u,\\
\partial_t b + u \cdot \nabla b - \nu \Delta b  =  b\cdot \nabla u,\\
\nabla \cdot u=0, \nabla \cdot b=0,\\
u(x, 0) = u_{0}(x), \omega(x, 0) = \omega_{0}(x), b(x, 0) = b_{0}(x).
\end{array} \right.
\end{align}
This paper first aims at the global stability problem of system \eqref{MMEZ}. To shed some light on the potential difficulties of this problem, we briefly review several facts on the inviscid and resistive MHD equations. Particularly, when we neglect the effect of the angular velocity of rotation of particles of the fluid, namely $\chi=0$, then \eqref{MMEZ} reduces to
\begin{align}\label{MHDZ}
\left\{ \begin{array}{lcl}
\partial_t u + u \cdot \nabla u  = - \nabla p + b\cdot \nabla b,\\
\partial_t b + u \cdot \nabla b - \nu \Delta b  =  b\cdot \nabla u,\\
\nabla \cdot u=0, \nabla \cdot b=0.
\end{array} \right.
\end{align}
The global well-posedness of solutions of system \eqref{MHDZ} even for small initial data are still left open. Recently, Chen, Zhang and Zhou \cite{CZZ22} first proved the global stability for \eqref{MHDZ} in the torus $\mathbb T^3$ when the magnetic field is close to a constant equilibrium state satisfying the Diophantine condition. Very recently, Xie, Jiu and Liu \cite{XJL24} further studied this system with background magnetic field satisfying the Diophantine condition and reduced the regularity requirement on the initial data of \cite{CZZ22}. For more results on the global stability problem of MHD equations with partial dissipation, one can refer to \cite{AZ17,CW11,LZ09,LXZ15,LZ14,PZZ18,RWXZ14,WZ20,YY22,Zhai23,Zhang16,ZZ23,ZZ18} and the references therein.



The first goal of this paper is to study the global stability problem of the magneto-micropolar equations with zero kinematic viscosity \eqref{MMEZ}. Taking advantage of the key observation of regarding the both linear terms $2\chi \nabla\times \omega$ and $2\chi \nabla\times u$ as the perturbations, we can overcome the difficulties of the lack of dissipation and establish the global stability of system \eqref{MMEZ}. More precisely, we have the following theorem.


\begin{thm}\label{main1}
Consider the system \eqref{MMEZ} with $\chi>0$, $\kappa\geq0$, $\eta>0$ and $\nu>0$. Suppose that $(u_0, \omega_0, b_0) \in H^3(\mathbb T^3)$ with $\nabla\cdot u_0=\nabla\cdot b_0 = 0$ and $\int_{\mathbb T^3} u_0 dx=\int_{\mathbb T^3} \omega_0 dx=\int_{\mathbb T^3} b_0 dx=0$. Then there exists a constant $\epsilon > 0$ such that, if
\begin{equation}\label{sma}
\|u_{0}\|_{H^3(\mathbb T^3)} + \|\omega_{0}\|_{H^3(\mathbb T^3)} + \|b_{0}\|_{H^3(\mathbb T^3)}\leq \epsilon,
\end{equation}
then system \eqref{MMEZ} has a unique global strong solution $(u, \omega, b)$ satisfying, for any $ t > 0$ and some positive constants $C$ and $C_0$,
\begin{equation}\label{m1}
\|u(t)\|_{H^3(\mathbb T^3)} + \|\omega(t)\|_{H^3(\mathbb T^3)} + \|b(t)\|_{H^3(\mathbb T^3)}\leq C\epsilon,
\end{equation}
\begin{equation}\label{m2}
\|u(t)\|_{H^{3}(\mathbb T^3)} + \|\omega(t)\|_{H^{3}(\mathbb T^3)} + \|b(t)\|_{H^{3}(\mathbb T^3)} \leq C e^{-C_0 t}.
\end{equation}

\end{thm}

\begin{rem}

(1) Theorem \ref{main1} reveals the stabilization effect of the presence of the microstructure which actually generates extra smoothing effect and stabilizes the flow. Especially the dynamics micro-rotation viscosity $\chi > 0$ is essential for the magneto-micropolar fluid, otherwise the global motion is unaffected by the micro-rotations, and \eqref{MMEZ} reduces to the inviscid and resistive MHD equations which global well-posed problem is still a challenging problem.



(2) Note that $\omega=(0, 0, \omega_3)$ for the 2D magneto-micropolar equations, so our methods used here also yield the global stability and exponential decay of solutions for the 2D magneto-micropolar equations \eqref{MMEZ} with minor technical revisions.

\end{rem}

When $\mu=0$, $\eta>0$ and $\nu=0$, system \eqref{MME} reduces to the following system
\begin{align}\label{MMEZZ}
\left\{ \begin{array}{lcl}
\partial_t u + u \cdot \nabla u - \chi \Delta u  = - \nabla p + 2\chi \nabla\times \omega + b\cdot \nabla b,\\
\partial_t \omega + u \cdot \nabla \omega + 4\chi \omega - \kappa \nabla\nabla\cdot \omega - \eta \Delta \omega = 2\chi \nabla\times u,\\
\partial_t b + u \cdot \nabla b  =  b\cdot \nabla u,\\
\nabla \cdot u=0, \nabla \cdot b=0,\\
u(x, 0) = u_{0}(x), \omega(x, 0) = \omega_{0}(x), b(x, 0) = b_{0}(x).
\end{array} \right.
\end{align}
If the micro-rotation viscosity $\chi=0$, then \eqref{MMEZZ} becomes the inviscid MHD equations
\begin{align}\label{MHDZZ}
\left\{ \begin{array}{lcl}
\partial_t u + u \cdot \nabla u  = - \nabla p + b\cdot \nabla b,\\
\partial_t b + u \cdot \nabla b  =  b\cdot \nabla u,\\
\nabla \cdot u=0, \nabla \cdot b=0.
\end{array} \right.
\end{align}
It is not known if smooth solutions of the 2D and 3D ideal MHD systems \eqref{MHDZZ} blow up in finite time. Using the Els\"{a}sser variable, Bardos, Sulem and Sulem \cite{BSS88} successfully proved the existence of global solutions of \eqref{MHDZZ} for small initial data in a weighted H\"{o}lder space subject to a strong magnetic field. Later, the global in time vanishing viscosity limit
of the full diffusive MHD system to the ideal equations under a strong magnetic field was further studied by Cai and Lei \cite{CL18}, He, Xu and Yu \cite{HXY18} and Wei and Zhang \cite{WZ17}.

The second goal of this paper is to study the global well-posedness of \eqref{MMEZZ} near a background magnetic field. More precisely, the background magnetic field refers to the special steady-state solution $(u^{(0)}, \omega^{(0)}, b^{(0)})$, where
$$
u^{(0)} = 0,\ \omega^{(0)} =  0,\ b^{(0)}= \alpha
$$
with $\alpha$ satisfying the Diophantine condition defined in Lemma \ref{DC} below. Any perturbation $(u, \omega, b)$ near $(u^{(0)}, \omega^{(0)}, b^{(0)})$ with $B=b-b^{(0)}$ is governed by
\begin{align}\label{MMEZZP}
\left\{ \begin{array}{lcl}
\partial_t u + u \cdot \nabla u - \chi \Delta u  = - \nabla p + 2\chi \nabla\times \omega + B\cdot \nabla B + \alpha\cdot \nabla B,\\
\partial_t \omega + u \cdot \nabla \omega + 4\chi \omega - \kappa \nabla\nabla\cdot \omega - \eta \Delta \omega = 2\chi \nabla\times u,\\
\partial_t B + u \cdot \nabla B  =  B\cdot \nabla u + \alpha\cdot \nabla u,\\
\nabla \cdot u=0, \nabla \cdot B=0,\\
u(x, 0) = u_{0}(x), \omega(x, 0) = \omega_{0}(x), B(x, 0) = B_{0}(x).
\end{array} \right.
\end{align}


Now we are in a position to state our second main result on the global stability and large time behavior of solutions of system \eqref{MMEZZP} as follows.

\begin{thm}\label{main2}
Consider the system \eqref{MMEZZP} with $|\alpha|^2<\chi<2$, $\kappa\geq0$ and $\eta>0$. Suppose that $(u_0, \omega_0, B_0) \in H^N(\mathbb T^3)$ with $N\geq 4r+11$ ($r>2$) and $\nabla\cdot u_0=\nabla\cdot B_0 = 0$. Assume also that $\int_{\mathbb T^3} u_0 dx=\int_{\mathbb T^3} \omega_0 dx=\int_{\mathbb T^3} B_0 dx=0$. Then there exists a constant $\epsilon > 0$ such that, if
\begin{equation}\label{Sma}
\|u_{0}\|_{H^N(\mathbb T^3)} + \|\omega_{0}\|_{H^N(\mathbb T^3)} + \|B_{0}\|_{H^N(\mathbb T^3)}\leq \epsilon,
\end{equation}
then system \eqref{MMEZZP} has a unique global solution $(u, \omega, B)$ satisfying for any $ t > 0$,
\begin{equation}\label{M1}
\|u(t)\|_{H^N(\mathbb T^3)} + \|\omega(t)\|_{H^N(\mathbb T^3)} + \|B(t)\|_{H^N(\mathbb T^3)}\leq C\epsilon,
\end{equation}
\begin{equation}\label{M2}
\|u(t)\|_{H^{r+5}(\mathbb T^3)} + \|\omega(t)\|_{H^{r+5}(\mathbb T^3)} + \|B(t)\|_{H^{r+5}(\mathbb T^3)} \leq C (1+t)^{-\frac32}.
\end{equation}

\end{thm}

\vskip .1in
\begin{rem}

(1) Theorem \ref{main2} also reveals the stabilization effect of the microstructure. If the micro-rotation effect of particles of the fluid is neglected, that is $\chi=0$, \eqref{MMEZZ} reduces to the inviscid MHD equations which yields the introduction of the weighted spaces to show the global well-posedness.

(2) The assumption of the structure condition $|\alpha|^2<\chi<2$ is due to the lack of kinematic viscosity and magnetic diffusion, and the interaction of the linear terms $2\chi \nabla\times \omega$, $2\chi \nabla\times u$ and the perturbation terms $\alpha\cdot \nabla B$, $\alpha\cdot \nabla u$. An interesting problem is whether or not this assumption can be removed.

\end{rem}

The rest of this paper is divided into two sections. Section 2 establishes Theorem \ref{main1} while Section 3 proves Theorem \ref{main2}. Let us end this section with some notations which shall be frequently used in this paper. To simplify the notations, we will write $\int f$ for $\int_{\mathbb T^3} f dx$, $\|f\|_{L^q}$ for $\|f\|_{L^q(\mathbb T^3)}$, $\|f\|_{\dot{H}^s}$ for $\|f\|_{\dot{H}^s(\mathbb T^3)}$, and $\|f\|_{H^s}$ for $\|f\|_{H^s(\mathbb T^3)}$.

\vskip .3in
\section{Proof of Theorem \ref{main1}}
\vskip .1in

This section proves Theorem \ref{main1}. Since the local well-posedness of \eqref{MMEZ} in $H^{3}$ follows from standard approach such as Friedrichs method, the crucial piece to prove Theorem \ref{main1} is to establish the global {\it{a priori}} bounds in $H^{3}$ for $(u, \omega, b)$.

\begin{proof}[Proof of Theorem \ref{main1}]

Standard calculations yield
\begin{align}\label{l2}
\frac{1}{2} \frac{d}{dt}\|(u, \omega, b)\|_{L^{2}}^{2} + \eta \|\nabla \omega\|_{L^{2}}^{2} + \nu \|\nabla b\|_{L^{2}}^{2}\leq 0.
\end{align}
Applying $\nabla^3$ to \eqref{MMEZ}, and taking the $L^2$-inner product to the resultants with $(\nabla^3 u, \nabla^3 \omega, \nabla^3 b)$ respectively, we derive that
\begin{equation}\label{t1}
\begin{aligned}
&\frac 12 \frac{d}{dt} \|(u, \omega, b)\|_{\dot{H}^3}^2 + \chi \|\nabla u\|_{\dot{H}^3}^2 + 4\chi \|\omega\|_{\dot{H}^3}^2 + \eta \|\nabla \omega\|_{\dot{H}^3}^2 + \nu \|\nabla b\|_{\dot{H}^3}^2\\
&\leq 4\chi \int \nabla^3 (\nabla\times u) \cdot \nabla^3 \omega + \int [\nabla^3, b\cdot \nabla]b \cdot \nabla^3 u - \int [\nabla^3, u\cdot \nabla]u \cdot \nabla^3 u\\
&\ \ \ - \int [\nabla^3, u\cdot \nabla]\omega \cdot \nabla^3 \omega + \int [\nabla^3, b\cdot \nabla]u \cdot \nabla^3 b - \int [\nabla^3, u\cdot \nabla]b \cdot \nabla^3 b\\
&:= K_1 + K_2 + K_3 + K_4 + K_5 + K_6.
\end{aligned}
\end{equation}
Taking advantage of Young's inequality, we get
\begin{align*}
K_1 \leq \chi \|\nabla u\|_{\dot{H}^3}^2 + 4\chi \|\omega\|_{\dot{H}^{3}}^2.
\end{align*}
It follows from H\"{o}lder's inequality and the commutator estimate of \cite{KPV91} that
\begin{align*}
K_2&\leq\|[\nabla^3, b\cdot \nabla]b\|_{L^2} \|\nabla^3 u\|_{L^2}\\
&\leq C \|\nabla b\|_{L^\infty} \|\nabla b\|_{\dot{H}^{2}} \|\nabla u\|_{\dot{H}^{2}}\\
&\leq C \|\nabla b\|_{H^{2}} \|(\nabla u, \nabla b)\|_{\dot{H}^{2}}^2.
\end{align*}
Similarly, we have
\begin{align*}
&K_3\leq C \|\nabla u\|_{H^{2}} \|\nabla u\|_{\dot{H}^{2}}^2,\\
&K_4\leq C \|(\nabla u, \nabla \omega)\|_{H^{2}} \|(\nabla u, \nabla \omega)\|_{\dot{H}^{2}}^2,\\
&K_5 + K_6\leq C \|(\nabla u, \nabla b)\|_{H^{2}} \|(\nabla u, \nabla b)\|_{\dot{H}^{2}}^2.
\end{align*}
Substituting the above bounds into \eqref{t1}, one arrives at
\begin{equation}\label{hs1}
\begin{aligned}
&\frac 12 \frac{d}{dt} \|(u, \omega, b)\|_{\dot{H}^3}^2 + \eta \|\nabla \omega\|_{\dot{H}^3}^2  + \nu \|\nabla b\|_{\dot{H}^3}^2\\
&\leq C \|(\nabla u, \nabla \omega, \nabla b)\|_{H^{2}} \|(\nabla u, \nabla \omega, \nabla b)\|_{\dot{H}^{2}}^2.
\end{aligned}
\end{equation}
Adding \eqref{l2} and \eqref{hs1} up, we have
\begin{equation}\label{hs}
\begin{aligned}
&\frac{d}{dt} (\|(u, \omega, b)\|_{H^3}^2 + 2\eta \|\nabla \omega\|_{H^3}^2 + 2\nu \|\nabla b\|_{H^3}^2\\
&\leq C \|(\nabla u, \nabla \omega, \nabla b)\|_{\dot{H}^{2}} \|(u, \omega, b)\|_{H^3}^2.
\end{aligned}
\end{equation}
Due to the lack of kinematic viscosity, which leads to no velocity dissipative effect in \eqref{hs}, we cannot show the global existence of solutions from \eqref{hs} directly. To overcome the difficulty, we turn to establish the exponential decay of $\|(\nabla u, \nabla \omega, \nabla b)(t)\|_{\dot{H}^{2}}$ or alternatively $\|(\nabla\times u, \nabla\times \omega, \nabla\times b)(t)\|_{\dot{H}^{2}}$ by fully exploiting the structure of system \eqref{MMEZ}. Applying $\nabla^2\nabla\times $ to \eqref{MMEZ}, then as in the estimate of \eqref{hs1}, we obtain
\begin{align}\label{t2}
&\frac{d}{dt} \|(\nabla\times u, \nabla\times \omega, \nabla\times b)\|_{\dot{H}^{2}}^2) + 2\eta \|\nabla\times \omega\|_{\dot{H}^{3}}^2 + 2\nu \|\nabla\times b\|_{\dot{H}^{3}}^2\nonumber\\
&\leq C \|(\nabla u, \nabla \omega, \nabla b)\|_{H^{2}} \|(\nabla\times u, \nabla\times \omega, \nabla\times b)\|_{\dot{H}^{2}}^2,
\end{align}
where we have also used the fact
\begin{equation*}
\begin{aligned}
\int \nabla\times (\nabla\nabla\cdot \omega)\cdot \nabla\times\omega
&=\int \varepsilon_{ijk}\partial_j\partial_k(\nabla\cdot \omega) \varepsilon_{ilm}\partial_l\omega_m\\
&=\int (\partial_j\partial_k(\nabla\cdot \omega) \partial_j\omega_k - \partial_j\partial_k(\nabla\cdot \omega) \partial_k\omega_j)\\
&=0.
\end{aligned}
\end{equation*}
Obviously, formula \eqref{t2} is not a closed differential inequality because of the lack of kinematic viscosity. This forces us to include suitable extra terms in this energy estimates. We discover that the term $\|\nabla\times u\|_{\dot{H}^2}^2$ serves our purpose perfectly, which comes from the equation of micro-rotational velocity field $\omega$ by regarding $2\chi\nabla\times u$ as a perturbation term. More precisely, applying $\nabla^2$ to the second equation of \eqref{MMEZ}, multiplying the result by $\nabla^2 (\nabla\times u)$ and integrating with respect to space domain, we get
\begin{equation}\label{t3}
\begin{aligned}
2\chi\|\nabla\times u\|_{\dot{H}^2}^2 &= \int \nabla^2 \omega_t \cdot \nabla^2 (\nabla\times u) + \int \nabla^2 (u\cdot\nabla \omega) \cdot \nabla^2 (\nabla\times u)\\
&\ \ \ + 4\chi \int \nabla^2 \omega \cdot \nabla^2 (\nabla\times u) - \kappa \int \nabla^2 (\nabla\nabla\cdot \omega) \cdot \nabla^2 (\nabla\times u)\\
&\ \ \  - \eta \int \nabla^2 (\Delta \omega) \cdot \nabla^2 (\nabla\times u)\\
&:= J_1 + J_2 + J_3 + J_4 + J_5.
\end{aligned}
\end{equation}
It does not appear possible to bound $J_1$ directly. Our strategy is to make use of the special structure of this term via the equation of velocity field $u$. This substitution generates more terms, but fortunately all the resulting terms can be bounded suitably. More precisely, we have
\begin{equation*}
\begin{aligned}
J_1 &= \frac{d}{dt}\int \nabla^2 \omega \cdot \nabla^2 (\nabla\times u) - \int \nabla^2 \omega \cdot \nabla^2 (\nabla\times u_t)\\
&=\frac{d}{dt}\int \nabla^2 \omega \cdot \nabla^2 (\nabla\times u) - 2\chi  \int \nabla^2 \omega \cdot \nabla^2 \nabla\times (\nabla\times \omega) - \int \nabla^2 \omega \cdot \nabla^2 \nabla\times (b\cdot\nabla b)\\
&\ \ \ + \int \nabla^2 \omega \cdot \nabla^2 \nabla\times (u\cdot\nabla u) - \chi \int \nabla^2 \omega \cdot \nabla^2 \nabla\times (\Delta u)\\
&:= \frac{d}{dt}\int \nabla^2 \omega \cdot \nabla^2 (\nabla\times u) + J_{11} + J_{12} + J_{13} + J_{14}.
\end{aligned}
\end{equation*}
The property of $\nabla\times$ implies
\begin{equation*}
\begin{aligned}
J_{11}=-2\chi \|\nabla\times \omega\|_{\dot{H}^2}^2.
\end{aligned}
\end{equation*}
Using H\"{o}lder's inequality, the commutator estimate and the Poincar\'{e} inequality, one finds
\begin{equation*}
\begin{aligned}
J_{12}
&= - \int \nabla^2 (\nabla\times \omega) \cdot \nabla^2 (b\cdot\nabla b)\\
&\leq \|\nabla^2 (\nabla\times \omega)\|_{L^2} \|\nabla^2 (b\cdot\nabla b)\|_{L^2}\\
&\leq C \|\nabla b\|_{H^{2}} \|(\nabla\times \omega, \nabla\times b)\|_{\dot{H}^{2}}^2
\end{aligned}
\end{equation*}
Similarly, we have
\begin{equation*}
\begin{aligned}
J_{13}
\leq C \|\nabla u\|_{H^{2}} \|(\nabla\times u, \nabla\times \omega)\|_{\dot{H}^{2}}^2.
\end{aligned}
\end{equation*}
Applying integration by parts and Young's inequality, we obtain
\begin{equation*}
\begin{aligned}
J_{14}\leq \frac{\chi}{8} \|\nabla\times u\|_{\dot{H}^2}^2 + C \|\nabla\times \omega\|_{\dot{H}^{3}}^2.
\end{aligned}
\end{equation*}
Substituting the bounds of $J_{11}$-$J_{14}$ into $J_1$, we derive that
\begin{equation*}
\begin{aligned}
J_{1}&\leq \frac{d}{dt}\int \nabla^2 \omega \cdot \nabla^2 (\nabla\times u) - 2\chi\|\nabla\times \omega\|_{\dot{H}^2}^2\\
&\ \ \ + C \|(\nabla u, \nabla b)\|_{H^{2}} \|(\nabla\times u, \nabla\times \omega, \nabla\times b)\|_{\dot{H}^{2}}^2\\
&\ \ \ + \frac{\chi}{8} \|\nabla\times u\|_{\dot{H}^2}^2 + C \|\nabla\times \omega\|_{\dot{H}^{3}}^2.
\end{aligned}
\end{equation*}
As in the estimate of $J_{12}$, we have
\begin{equation*}
\begin{aligned}
J_2\leq C \|(\nabla u, \nabla \omega)\|_{H^{2}} \|(\nabla\times u, \nabla\times \omega)\|_{\dot{H}^{2}}^2.
\end{aligned}
\end{equation*}
Applying Young's inequality and the Poincar\'{e} inequality to yield
\begin{equation*}
\begin{aligned}
J_3&\leq \frac{\chi}{8} \|\nabla^2 (\nabla\times u)\|_{L^2}^2 + C \|\nabla^2 \omega\|_{L^2}^2\\
&\leq \frac{\chi}{8} \|\nabla\times u\|_{\dot{H}^2}^2 + C \|\nabla\times \omega\|_{\dot{H}^{3}}^2.
\end{aligned}
\end{equation*}
Similarly,
\begin{equation*}
\begin{aligned}
J_4+J_5\leq\frac{\chi}{8} \|\nabla\times u\|_{\dot{H}^2}^2 + C \|\nabla\times \omega\|_{\dot{H}^{3}}^2.
\end{aligned}
\end{equation*}
Incorporating the above bounds into \eqref{t3}, one infers that
\begin{equation}\label{t4}
\begin{aligned}
&\chi\|\nabla\times u\|_{\dot{H}^2}^2 + 2\chi\|\nabla\times \omega\|_{\dot{H}^2}^2\\
&\leq \frac{d}{dt}\int \nabla^2 \omega \cdot \nabla^2 (\nabla\times u) + C_1 \|\nabla\times \omega\|_{\dot{H}^{3}}^2\\
&\ \ \ + C \|(\nabla u, \nabla \omega, \nabla b)\|_{H^{2}} \|(\nabla\times u, \nabla\times \omega, \nabla\times b)\|_{\dot{H}^{2}}^2.
\end{aligned}
\end{equation}
Multiplying \eqref{t2} by $A$ with $A>1$ to be chosen later and adding the result to \eqref{t4}, we have
\begin{equation}\label{t5}
\begin{aligned}
&\frac{d}{dt} \Big(A\|(\nabla\times u, \nabla\times \omega, \nabla\times b)\|_{\dot{H}^{2}}^2) - \int \nabla^2 \omega \cdot \nabla^2 (\nabla\times u) \Big)\\
&\ \ \ + (2 A \eta - C_1) \|\nabla\times \omega\|_{\dot{H}^{3}}^2 + 2 A \nu \|\nabla\times b\|_{\dot{H}^{3}}^2 + \chi\|\nabla\times u\|_{\dot{H}^{2}}^2\\
&\leq C_2 A \|(\nabla u, \nabla \omega, \nabla b)\|_{H^{2}} \|(\nabla\times u, \nabla\times \omega, \nabla\times b)\|_{\dot{H}^{2}}^2.
\end{aligned}
\end{equation}
In the following, with the above estimates at our disposals, we shall complete the proof by using the bootstrapping argument. Let
$$
\mathcal E(t)=\sup_{0\leq\tau\leq t} \|(u, \omega, b)(\tau)\|_{H^3}^2 + \int_{0}^{t} \|(\nabla \omega, \nabla b)(\tau)\|_{H^3}^2 d\tau.
$$
Then for any $t>0$, we obtain from \eqref{hs},
\begin{align}\label{e}
\mathcal E(t)\leq \mathcal E(0) + C_3 \mathcal E(t) \int_0^t \|(\nabla u, \nabla \omega, \nabla b)(\tau)\|_{\dot{H}^{2}} d\tau.
\end{align}
Then we make the ansatz that, for $t\in[0,T ]$ with $T > 0$,
\begin{equation}\label{e1}
\mathcal E(t)\leq 6\epsilon^2.
\end{equation}
Let
$$
\mathcal F(t)=A \|(\nabla\times u, \nabla\times \omega, \nabla\times b)(t)\|_{\dot{H}^{2}}^2 - \int \nabla^2 \omega \cdot \nabla^2 (\nabla\times u).
$$
Taking $A>\frac{C_1}{\eta}$ large enough, together with the Poincar\'{e} inequality, we have
$$
\mathcal F(t)\geq \|(\nabla\times u, \nabla\times \omega, \nabla\times b)(t)\|_{\dot{H}^{2}}^2.
$$
Inserting \eqref{e1} into \eqref{t5}, using the Poincar\'{e} inequality, and choosing $\epsilon$ sufficiently small, we obtain
\begin{equation}\label{hs2}
\begin{aligned}
&\frac{d}{dt}\mathcal F(t) + C_4 \mathcal F(t)\leq 0,
\end{aligned}
\end{equation}
which implies
\begin{equation}\label{t6}
\begin{aligned}
\mathcal F(t) \leq C \epsilon^2 e^{-C_4t}.
\end{aligned}
\end{equation}
Then it follows from \eqref{t6} and $\nabla\cdot u=\nabla\cdot b=0$ that
\begin{equation}\label{t7}
\begin{aligned}
\|(\nabla u, \nabla b)(t)\|_{\dot{H}^{2}}^2 + \|(\nabla\times \omega)(t)\|_{\dot{H}^{2}}^2 \leq C \epsilon^2 e^{-C_4t}.
\end{aligned}
\end{equation}
Note that
$$
\|\nabla\omega\|_{\dot{H}^{2}}\leq C(\|\nabla\times \omega\|_{\dot{H}^{2}} + \|\nabla\cdot \omega\|_{\dot{H}^{2}}),
$$
so, to close the estimates, we also need to control $\|\nabla\cdot \omega\|_{\dot{H}^{2}}$. To this end, applying $\nabla^2 \nabla\cdot$ to  the second equation of \eqref{MMEZ}, and taking the $L^2$-inner product to the resultant with $\nabla^2 \nabla\cdot \omega$, we obtain
\begin{equation}\label{t8}
\begin{aligned}
&\frac 12 \frac{d}{dt} \|\nabla\cdot\omega\|_{\dot{H}^{2}}^2 + 4\chi \|\nabla\cdot\omega\|_{\dot{H}^{2}}^2 + \eta \|\nabla\cdot\omega\|_{\dot{H}^{3}}^2\\
&\leq - \int \nabla^2 \nabla\cdot (u\cdot \nabla \omega)\cdot \nabla^2 \nabla\cdot\omega\\
&= - \int \nabla^2 u\cdot \nabla \nabla\cdot\omega \cdot \nabla^2 \nabla\cdot\omega - 2 \int \nabla u\cdot \nabla (\nabla\nabla\cdot\omega) \cdot \nabla^2 \nabla\cdot\omega\\
&\ \ \ - \int \nabla u\cdot \nabla \nabla^2 \omega \cdot \nabla^2 \nabla\cdot\omega - \int \nabla^3 u\cdot \nabla \omega \cdot \nabla^2 \nabla\cdot\omega\\
&:= R_1 + R_2 + R_3 + R_4.
\end{aligned}
\end{equation}
Using integration by parts, \eqref{e1}, \eqref{t7} and the Poincar\'{e} inequality, we get
\begin{equation*}
\begin{aligned}
R_1&=\int \nabla u\cdot \nabla (\nabla\nabla\cdot\omega) \cdot \nabla^2 \nabla\cdot\omega + \int \nabla u\cdot \nabla \nabla\cdot\omega \cdot \nabla^3 \nabla\cdot\omega\\
&\leq \|\nabla u\|_{L^{\infty}} (\|\nabla (\nabla\nabla\cdot\omega)\|_{L^2}\|\nabla^2 \nabla\cdot\omega\|_{L^2} + \|\nabla \nabla\cdot\omega\|_{L^2}\|\nabla^3 \nabla\cdot\omega\|_{L^2})\\
&\leq C \epsilon \|\nabla\cdot\omega\|_{\dot{H}^{3}}^2.
\end{aligned}
\end{equation*}
Similarly,
\begin{equation*}
\begin{aligned}
R_2\leq C \epsilon \|\nabla\cdot\omega\|_{\dot{H}^{3}}^2.
\end{aligned}
\end{equation*}
Applying H\"{o}lder's inequality, \eqref{e1} and the Poincar\'{e} inequality to yield
\begin{equation*}
\begin{aligned}
R_3&\leq \|\nabla u\|_{L^{\infty}} \|\nabla \nabla^2 \omega\|_{L^2}\|\nabla^2 \nabla\cdot\omega\|_{L^2}\\
&\leq \|\nabla u\|_{L^{\infty}} (\|\nabla^2 \nabla\times\omega\|_{L^2}\|\nabla^2 \nabla\cdot\omega\|_{L^2} + \|\nabla^2 \nabla\cdot\omega\|_{L^2}^2)\\
&\leq C \epsilon^2 e^{-\frac{C_4}{2}t} + C \epsilon \|\nabla\cdot\omega\|_{\dot{H}^{3}}^2.
\end{aligned}
\end{equation*}
Using the Gagliardo-Nirenberg inequality and Young's inequality, one has
\begin{equation*}
\begin{aligned}
R_4&\leq \|\nabla^3 u\|_{L^2} \|\nabla \omega\|_{L^4}\|\nabla^2 \nabla\cdot\omega\|_{L^4}\\
&\leq \|\nabla^3 u\|_{L^2} \|\nabla \omega\|_{L^2}^{\frac14} \|\nabla^2 \omega\|_{L^2}^{\frac34} \|\nabla^2 \nabla\cdot\omega\|_{L^2}^{\frac14} \|\nabla^3 \nabla\cdot\omega\|_{L^2}^{\frac34}\\
&\leq C \epsilon^2 e^{-\frac{C_4}{2}t} + C \epsilon \|\nabla\cdot\omega\|_{\dot{H}^{3}}^2.
\end{aligned}
\end{equation*}
Inserting the above bounds into \eqref{t8}, and choosing $\epsilon$ sufficiently small, one arrives at
\begin{equation}\label{t9}
\begin{aligned}
\frac{d}{dt} \|\nabla\cdot\omega\|_{\dot{H}^{2}}^2 + 8\chi \|\nabla\cdot\omega\|_{\dot{H}^{2}}^2 + \eta \|\nabla\cdot\omega\|_{\dot{H}^{3}}^2\leq C \epsilon^2 e^{-\frac{C_4}{2}t}.
\end{aligned}
\end{equation}
Applying the Gronwall's inequality, we obtain
\begin{equation}\label{t10}
\begin{aligned}
\|\nabla\cdot\omega\|_{\dot{H}^{2}}^2\leq C \epsilon^2 e^{-C_5t}.
\end{aligned}
\end{equation}
Then \eqref{t7} and \eqref{t10} imply
\begin{equation*}
\begin{aligned}
\|(\nabla u, \nabla\omega, \nabla b)(t)\|_{\dot{H}^{2}}^2 + \|(\nabla\times \omega)(t)\|_{\dot{H}^{2}}^2 \leq C_6 \epsilon^2 e^{-C_7t}.
\end{aligned}
\end{equation*}
Inserting this estimate, \eqref{sma} and \eqref{e1} into \eqref{e}, one arrives at for all $t\in[0,T]$,
\begin{align*}
\mathcal E(t)\leq \epsilon^2 + \frac{2C_1\sqrt{C_6}}{C_7} \epsilon^3\leq 3\epsilon^2,
\end{align*}
which is achievable by taking $\epsilon$ sufficiently small such that $\frac{2C_1\sqrt{C_6}}{C_7} \epsilon\leq2$. Thus the argument is closed and we complete the proof of Theorem \ref{main1}.

\end{proof}

\vskip .3in
\section{Proof of Theorem \ref{main2}}
\vskip .1in

This section is devoted to the proof of Theorem \ref{main2}. As preparations, we first recall the following lemma about the Diophantine condition which is satisfied for almost all vectors $\alpha\in\mathbb R^3$ as demonstrated in \cite{CZZ22}.

\begin{lemma}\label{DC}
Suppose that $\alpha\in\mathbb R^3$ satisfies the Diophantine condition, that is, for any $k\in\mathbb Z^3\setminus\{0\}$, there exist constants $c>0$ and $r>2$ such that
$$
|\alpha\cdot k|\geq \frac{c}{|k|^r}.
$$
Then it holds that for any $s\in\mathbb R$,
$$
\|f\|_{H^s(\mathbb T^3)}\leq C \|\alpha\cdot \nabla f\|_{H^{s+r}(\mathbb T^3)}
$$
with $\int_{\mathbb T^3} f dx=0$.

\end{lemma}

Now we are ready to prove Theorem \ref{main2}.

\begin{proof}[Proof of Theorem \ref{main2}]

The framework of the proof is the bootstrapping argument. The initial data $(u_0, \omega_0, B_{0})$ is assumed to be small, in the sense that
\begin{equation*}
\|(u_{0}, \omega_{0}, B_{0})\|_{H^N}\leq \epsilon
\end{equation*}
for some sufficiently small $\varepsilon>0$. Let $(u, \omega, B)$ be the corresponding solution. We make
the ansatz that, for $t\in [0, T]$ with $T>0$,
\begin{equation}\label{g1}
\|(u, \omega, B)(t)\|_{H^N}\leq \delta
\end{equation}
for $0<\delta<1$ to be determined later. Our main efforts are then devoted to proving the improved inequality, for all $t\in [0, T]$,
\begin{equation}\label{g2}
\|(u, \omega, B)(t)\|_{H^N}\leq \frac{\delta}{2}.
\end{equation}
Then the bootstrapping argument implies $T=\infty$ and that (\ref{g2}) actually holds for any $t<\infty$.

The rest of proof is devoted to showing (\ref{g2}). The proof is slightly long.  For the sake of clarity, we divide it into three steps.

\vskip .1in
\textbf{Step I. Estimate of $\|(u, \omega, B)\|_{H^{r+5}}$.}

Taking the $L^2$-inner products with $(u, \omega, B)$ to \eqref{MMEZZP}, one has
\begin{equation}\label{L2}
\begin{aligned}
\frac{1}{2} \frac{d}{dt} \|(u, \omega, B)\|_{L^{2}}^{2} + \kappa \|\nabla\cdot \omega\|_{L^{2}}^{2} + \eta \|\nabla \omega\|_{L^{2}}^{2}\leq 0,
\end{aligned}
\end{equation}
where we have used the facts
\begin{equation*}
\begin{aligned}
&\int B\cdot \nabla B \cdot u + \int B\cdot \nabla u \cdot B = 0,\ \int \alpha\cdot \nabla B \cdot u + \int \alpha\cdot \nabla u \cdot B = 0,\\
&4\chi\int \nabla\times u\cdot \omega \leq \chi \|\nabla u\|_{L^{2}}^{2} + 4\chi \|\omega\|_{L^{2}}^{2}.
\end{aligned}
\end{equation*}
Taking $\nabla^{r+5}$ to \eqref{MMEZZP}, and then making the $L^2$-inner product to resulting equations with $(\nabla^{r+5} u, \nabla^{r+5} \omega, \nabla^{r+5} B)$ respectively, we have
\begin{equation}\label{T3}
\begin{aligned}
&\frac 12 \frac{d}{dt} \|(u, \omega, B)\|_{\dot{H}^{r+5}}^2 + \chi \|\nabla u\|_{\dot{H}^{r+5}}^2 + 4\chi \|\omega\|_{\dot{H}^{r+5}}^2 + \kappa \|\nabla\cdot \omega\|_{\dot{H}^{r+5}}^2 + \eta \|\nabla \omega\|_{\dot{H}^{r+5}}^2\\
&\leq 4\chi \int \nabla^{r+5} (\nabla\times u) \cdot \nabla^{r+5} \omega + \int \nabla^{r+5} (B\cdot \nabla B) \cdot \nabla^{r+5} u + \int \nabla^{r+5} (\alpha\cdot \nabla B) \cdot \nabla^{r+5} u\\
&\ \ \ - \int [\nabla^{r+5}, u\cdot \nabla]u \cdot \nabla^{r+5} u - \int [\nabla^{r+5}, u\cdot \nabla]\omega \cdot \nabla^{r+5} \omega + \int \nabla^{r+5} (B\cdot \nabla u) \cdot \nabla^{r+5} B\\
&\ \ \  + \int \nabla^{r+5} (\alpha\cdot \nabla u) \cdot \nabla^{r+5} B - \int [\nabla^{r+5}, u\cdot \nabla]B \cdot \nabla^{r+5} B\\
&:= L_1 + L_2 + L_3 + L_4 + L_5 + L_6 + L_7 + L_8.
\end{aligned}
\end{equation}
Applying Young's inequality to yield
\begin{align*}
L_1 \leq \chi \|\nabla u\|_{\dot{H}^{r+5}}^2 + 4\chi \|\omega\|_{\dot{H}^{r+5}}^2.
\end{align*}
Using the commutator estimate and the Sobolev embedding theorem, one finds
\begin{align*}
L_5&\leq\|[\nabla^{r+5}, u\cdot \nabla]\omega\|_{L^2} \|\nabla^{r+5} \omega\|_{L^2}\\
&\leq C(\|\nabla u\|_{L^\infty} \|\omega\|_{\dot{H}^{r+5}}^2 + \|\nabla \omega\|_{L^\infty} \|u\|_{\dot{H}^{r+5}}\|\omega\|_{\dot{H}^{r+5}})\\
&\leq C \|(\nabla u, \nabla \omega)\|_{H^{2}} \|(u, \omega)\|_{\dot{H}^{r+5}}^2.
\end{align*}
Similarly, we have
\begin{align*}
L_4\leq C \|\nabla u\|_{H^{2}} \|u\|_{\dot{H}^{r+5}}^2.
\end{align*}
\begin{align*}
L_8\leq C \|(\nabla u, \nabla B)\|_{H^{2}} \|(u, B)\|_{\dot{H}^{r+5}}^2.
\end{align*}
Note that
\begin{align*}
L_2+L_6&=\int [\nabla^{r+5}, B\cdot \nabla]B \cdot \nabla^{r+5} u + \int [\nabla^{r+5}, B\cdot \nabla]u \cdot \nabla^{r+5} B\\
&\leq \|[\nabla^{r+5}, B\cdot \nabla]B\|_{L^2} \|\nabla^{r+5} u\|_{L^2} + \|[\nabla^{r+5}, B\cdot \nabla]u\|_{L^2} \|\nabla^{r+5} B\|_{L^2},
\end{align*}
then as in the estimate of $L_5$, we obtain
\begin{align*}
L_2+L_6\leq C \|(\nabla u, \nabla B)\|_{H^{2}} \|(u, B)\|_{\dot{H}^{r+5}}^2.
\end{align*}
It is easy to see that
\begin{align*}
L_3+L_7=0.
\end{align*}
Inserting the above estimates into \eqref{T3}, we get
\begin{equation}\label{Hs1}
\begin{aligned}
&\frac 12 \frac{d}{dt} \|(u, \omega, B)\|_{\dot{H}^{r+5}}^2 + \kappa \|\nabla\cdot \omega\|_{\dot{H}^{r+5}}^2 + \eta \|\nabla \omega\|_{\dot{H}^{r+5}}^2\\
&\leq C \|(\nabla u, \nabla \omega, \nabla B)\|_{H^{2}} \|(u, \omega, B)\|_{\dot{H}^{r+5}}^2.
\end{aligned}
\end{equation}
Adding \eqref{L2} and \eqref{Hs1} up, one arrives at
\begin{equation}\label{Hs}
\begin{aligned}
&\frac{d}{dt} \|(u, \omega, B)\|_{H^{r+5}}^2 + 2\kappa \|\nabla\cdot \omega\|_{H^{r+5}}^2 + 2\eta \|\nabla \omega\|_{H^{r+5}}^2\\
&\leq C \|(u, \omega, B)\|_{H^{3}} \|(u, \omega, B)\|_{H^{r+5}}^2.
\end{aligned}
\end{equation}

\vskip .1in
\textbf{Step II. Enhanced dissipation $\|u\|_{H^{r+5}}^2 + \|\alpha\cdot\nabla B\|_{H^{r+3}}^2$.}

This crucial estimate will help us to overcome the difficulties of lack of kinematic viscosity and magnetic diffusion, and eventually makes \eqref{Hs} become a closed differential inequality. To this end, applying $\nabla^{k}$ with $1\leq k\leq r+4$ to the second equation of \eqref{MMEZZP}, multiplying the result by $\nabla^{k} (\nabla\times u)$ and integrating with respect to space domain, yields
\begin{equation}\label{T4}
\begin{aligned}
2\chi\|\nabla^{k} (\nabla\times u)\|_{L^2}^2 &= \int \nabla^{k} \omega_t \cdot \nabla^{k} (\nabla\times u) + \int \nabla^{k} (u\cdot\nabla \omega) \cdot \nabla^{k} (\nabla\times u)\\
&\ \ \ + 4\chi \int \nabla^{k} \omega \cdot \nabla^{k} (\nabla\times u) - \kappa \int \nabla^{k} (\nabla\nabla\cdot \omega) \cdot \nabla^{k} (\nabla\times u)\\
&\ \ \  - \eta \int \nabla^{k} \Delta \omega \cdot \nabla^{k} (\nabla\times u)\\
&:= I_1 + I_2 + I_3 + I_4 + I_5.
\end{aligned}
\end{equation}
To estimate $I_1$, we use the equation of $u$ to obtain that
\begin{equation*}
\begin{aligned}
I_1 &= \frac{d}{dt}\int \nabla^{k} \omega \cdot \nabla^{k} (\nabla\times u) - \int \nabla^{k} \omega \cdot \nabla^{k} (\nabla\times u_t)\\
&= \frac{d}{dt}\int \nabla^{k} \omega \cdot \nabla^{k} (\nabla\times u) - \int \nabla^{k} \omega \cdot \nabla^{k} \nabla\times (-\nabla p + 2\chi \nabla\times \omega\\
&\ \ \ + B\cdot\nabla B + \alpha\cdot\nabla B - u\cdot\nabla u + \chi \Delta u)\\
&=\frac{d}{dt}\int \nabla^{k} \omega \cdot \nabla^{k} (\nabla\times u) - 2\chi  \int \nabla^{k} \omega \cdot \nabla^{k} \nabla\times (\nabla\times \omega)\\
&\ \ \ - \int \nabla^{k} \omega \cdot \nabla^{k} \nabla\times (B\cdot\nabla B) - \int \nabla^{k} \omega \cdot \nabla^{k} \nabla\times (\alpha\cdot\nabla B)\\
&\ \ \ + \int \nabla^{k} \omega \cdot \nabla^{k} \nabla\times (u\cdot\nabla u) - \chi \int \nabla^{k} \omega \cdot \nabla^{k} \nabla\times (\Delta u)\\
&:= \frac{d}{dt}\int \nabla^{k} \omega \cdot \nabla^{k} (\nabla\times u) + I_{11} + I_{12} + I_{13} + I_{14} + I_{15}.
\end{aligned}
\end{equation*}
It is easy to see that
\begin{equation*}
\begin{aligned}
I_{11}=-2\chi\|\nabla^{k} (\nabla\times \omega)\|_{L^2}^2.
\end{aligned}
\end{equation*}
Using H\"{o}lder's inequality and the commutator estimate, together with \eqref{g1}, we get
\begin{equation*}
\begin{aligned}
I_{12}
&= - \int \nabla^{k} (\nabla\times \omega) \cdot \nabla^{k} (B\cdot\nabla B)\\
&\leq \|\nabla^{k} (\nabla\times \omega)\|_{L^2} \|\nabla^{k} (B\cdot\nabla B)\|_{L^2}\\
&\leq C \|\nabla^{k} (\nabla\times \omega)\|_{L^{2}} \|B\|_{L^\infty} \|B\|_{\dot{H}^{k+1}}\\
&\leq \chi \|\nabla^{k} (\nabla\times \omega)\|_{L^{2}}^2 + C\delta^2 \|B\|_{H^2}^2.
\end{aligned}
\end{equation*}
By integration by parts and using Young's inequality, we derive that
\begin{equation*}
\begin{aligned}
I_{13}
&= - \int \nabla^{k+2} \omega \cdot \nabla^{k-2} \nabla\times (\alpha\cdot\nabla B)\\
&\leq C\|\nabla^{k+2} \omega\|_{L^2} \|\nabla^{k-1} (\alpha\cdot\nabla B)\|_{L^2}\\
&\leq C\|\nabla \omega\|_{\dot{H}^{k+1}}^2 + \frac{2-\chi}{16} \|\alpha\cdot\nabla B\|_{\dot{H}^{k-1}}^2.
\end{aligned}
\end{equation*}
Similarly, we have
\begin{equation*}
\begin{aligned}
I_{14}
&\leq \|\nabla^{k+1} \omega\|_{L^2} \|\nabla^{k}\nabla\times (u\otimes u)\|_{L^2}\\
&\leq \frac{\chi}{16} \|\nabla^{k} (\nabla\times u)\|_{L^{2}}^2 + C\delta^2 \|\omega\|_{\dot{H}^{k+1}}^2.
\end{aligned}
\end{equation*}
\begin{equation*}
\begin{aligned}
I_{15}
\leq \frac{\chi}{16} \|\nabla^{k} (\nabla\times u)\|_{L^{2}}^2 + C \|\nabla \omega\|_{\dot{H}^{k+1}}^2.
\end{aligned}
\end{equation*}
Substituting the bounds of $I_{11}$-$I_{15}$ into $I_1$, together with the Poincar\'{e} inequality, one arrives at
\begin{equation*}
\begin{aligned}
I_{1}&\leq \frac{d}{dt}\int \nabla^{k} \omega \cdot \nabla^{k} (\nabla\times u) - \chi\|\nabla^{k} (\nabla\times \omega)\|_{L^2}^2\\
&\ \ \ + C \delta^2 \|B\|_{H^2}^2 + C \|\nabla \omega\|_{\dot{H}^{k+1}}^2 + \frac{2-\chi}{16}  \|\alpha\cdot\nabla B\|_{\dot{H}^{k-1}}^2 + \frac{\chi}{8} \|\nabla^{k} (\nabla\times u)\|_{L^{2}}^2.
\end{aligned}
\end{equation*}
Applying the commutator estimate and Young's inequality, together with \eqref{g1}, we obtain
\begin{equation*}
\begin{aligned}
I_2
&\leq C \|\nabla^{k} (u\cdot\nabla \omega)\|_{L^{2}} \|\nabla^{k} (\nabla\times u)\|_{L^2}\\
&\leq C (\|u\|_{L^\infty} \|\omega\|_{\dot{H}^{k+1}} + \|\omega\|_{L^\infty} \|u\|_{\dot{H}^{k+1}}) \|\nabla^{k} (\nabla\times u)\|_{L^2}\\
&\leq \frac{\chi}{16} \|\nabla^{k} (\nabla\times u)\|_{L^{2}}^2 + C \delta^2 \|\omega\|_{\dot{H}^{k+1}}^2 + C \delta \|\nabla^{k} (\nabla\times u)\|_{L^{2}}^2.
\end{aligned}
\end{equation*}
It follows from Young's inequality that
\begin{equation*}
\begin{aligned}
I_3\leq \frac{\chi}{16} \|\nabla^{k} (\nabla\times u)\|_{L^2}^2 + C \|\omega\|_{\dot{H}^{k}}^2.
\end{aligned}
\end{equation*}
\begin{equation*}
\begin{aligned}
I_4+I_5\leq \frac{\chi}{16} \|\nabla^{k} (\nabla\times u)\|_{L^2}^2 + C \|\nabla\omega\|_{\dot{H}^{k+1}}^2.
\end{aligned}
\end{equation*}
Substituting the above bounds into \eqref{T4}, using the Poincar\'{e} inequality and choosing $\delta$ sufficiently small, we get
\begin{equation}\label{T5}
\begin{aligned}
&\frac{3}{2}\chi\|\nabla^{k} (\nabla\times u)\|_{L^2}^2 + \chi\|\nabla^{k} (\nabla\times \omega)\|_{L^2}^2 - \frac{d}{dt}\int \nabla^{k} \omega \cdot \nabla^{k} (\nabla\times u)\\
&\leq C \delta^2 \|B\|_{H^2}^2 + C \|\nabla \omega\|_{\dot{H}^{k+1}}^2 + \frac{2-\chi}{16}  \|\alpha\cdot\nabla B\|_{\dot{H}^{k-1}}^2.
\end{aligned}
\end{equation}
As in the estimate of \eqref{T5}, we derive that
\begin{equation}\label{T6}
\begin{aligned}
&\frac{3}{2}\chi\|\nabla\times u\|_{L^2}^2 + \chi\|\nabla\times \omega\|_{L^2}^2 - \frac{d}{dt}\int \omega \cdot (\nabla\times u)\\
&\leq C \delta^2 \|B\|_{H^2}^2 + C \|\nabla \omega\|_{\dot{H}^{1}}^2 + \frac{2-\chi}{16}  \|\alpha\cdot\nabla B\|_{L^2}^2.
\end{aligned}
\end{equation}
Combining \eqref{T5} and \eqref{T6} together, and using Lemma \ref{DC}, we obtain
\begin{equation}\label{T7}
\begin{aligned}
&\frac{3}{2}\chi\|\nabla\times u\|_{H^{r+4}}^2 + \chi \|\nabla\times \omega\|_{H^{r+4}}^2 - \frac{d}{dt}\sum_{k=0}^{r+4}\int \nabla^{k} \omega \cdot \nabla^{k} (\nabla\times u)\\
&\leq \Big(C \delta^2 + \frac{2-\chi}{16}\Big) \|\alpha\cdot\nabla B\|_{H^{r+3}}^2 + C \|\nabla \omega\|_{H^{r+5}}^2.
\end{aligned}
\end{equation}

Subsequently, to close the estimate, we take $\nabla^{k}$ with $1\leq k\leq r+3$ to the first equation of \eqref{MMEZZP}, multiply the result by $\nabla^{k} (\alpha\cdot\nabla B)$ and integrate the resultant over $\mathbb T^3$ to obtain that
\begin{equation}\label{T8}
\begin{aligned}
\|\nabla^{k} (\alpha\cdot\nabla B)\|_{L^2}^2 &= \int \nabla^{k} u_t \cdot \nabla^{k} (\alpha\cdot\nabla B) + \int \nabla^{k} (u\cdot\nabla u) \cdot \nabla^{k} (\alpha\cdot\nabla B)\\
&\ \ \ - \chi \int \nabla^{k} (\Delta u) \cdot \nabla^{k} (\alpha\cdot\nabla B) - 2\chi \int \nabla^{k} (\nabla\times\omega) \cdot \nabla^{k} (\alpha\cdot\nabla B)\\
&\ \ \ - \int \nabla^{k} (B\cdot\nabla B) \cdot \nabla^{k} (\alpha\cdot\nabla B) \\
&:= M_1 + M_2 + M_3 + M_4 + M_5.
\end{aligned}
\end{equation}
As in the estimate of $I_1$, we use the special structure of the equation for $B$ in \eqref{MMEZZP} and make the
substitution $\partial_t B = B\cdot \nabla u + \alpha\cdot \nabla u - u \cdot \nabla B$. Then $M_1$ can be written as
\begin{equation*}
\begin{aligned}
M_1 &= \frac{d}{dt}\int \nabla^{k} u \cdot \nabla^{k} (\alpha\cdot\nabla B) - \int \nabla^{k} u \cdot \nabla^{k} (\alpha\cdot\nabla B_t)\\
&=\frac{d}{dt}\int \nabla^{k} u \cdot \nabla^{k} (\alpha\cdot\nabla B) + \int \nabla^{k} u \cdot \nabla^{k} (\alpha\cdot\nabla (B\cdot\nabla u))\\
&\ \ \ + \int \nabla^{k} u \cdot \nabla^{k} (\alpha\cdot\nabla (\alpha\cdot\nabla u)) - \int \nabla^{k} u \cdot \nabla^{k} (\alpha\cdot\nabla (u\cdot\nabla B))\\
&:= \frac{d}{dt}\int \nabla^{k} u \cdot \nabla^{k} (\alpha\cdot\nabla B) + M_{11} + M_{12} + M_{13}.
\end{aligned}
\end{equation*}
By integration by parts, using the commutator estimate, \eqref{g1} and the Poincar\'{e} inequality, one gets
\begin{equation*}
\begin{aligned}
M_{11}
&\leq \|\nabla^{k+2} u\|_{L^2} \|\nabla^{k-2} (\alpha\cdot\nabla (B\cdot\nabla u))\|_{L^2}\\
&\leq C \|\nabla^{k+2} u\|_{L^2} (\|\nabla^{k-1} B\|_{L^\infty} \|\nabla u\|_{L^2} + \|B\|_{L^\infty} \|\nabla^{k} u\|_{L^2})\\
&\leq C \delta \|\nabla u\|_{\dot{H}^{k+1}}^2.
\end{aligned}
\end{equation*}
Similarly, we have
\begin{equation*}
\begin{aligned}
M_{13}\leq C \delta \|\nabla u\|_{\dot{H}^{k+1}}^2.
\end{aligned}
\end{equation*}
Applying $\nabla\cdot u=0$ to yield
\begin{equation*}
\begin{aligned}
M_{12}\leq |\alpha|^2 \|\nabla u\|_{\dot{H}^{k}}^2.
\end{aligned}
\end{equation*}
Thus we obtain
\begin{equation*}
\begin{aligned}
M_1\leq \frac{d}{dt}\int \nabla^{k} u \cdot \nabla^{k} (\alpha\cdot\nabla B) + C \delta \|\nabla u\|_{\dot{H}^{k+1}}^2 + |\alpha|^2 \|\nabla u\|_{\dot{H}^{k}}^2.
\end{aligned}
\end{equation*}
As in the estimate of $I_2$,
\begin{equation*}
\begin{aligned}
M_2&\leq C \|\nabla^{k} (u\cdot\nabla u)\|_{L^{2}} \|\nabla^{k} (\alpha\cdot\nabla B)\|_{L^2}\\
&\leq C \|u\|_{L^\infty} \|u\|_{\dot{H}^{k+1}} \|\nabla^{k} (\alpha\cdot\nabla B)\|_{L^2}\\
&\leq \frac{2-\chi}{16} \|\nabla^{k} (\alpha\cdot\nabla B)\|_{L^2}^2 + C \delta^2 \|\nabla u\|_{\dot{H}^{k+1}}^2.
\end{aligned}
\end{equation*}
Similarly, we have
\begin{equation*}
\begin{aligned}
M_5\leq \frac{2-\chi}{16} \|\nabla^{k} (\alpha\cdot\nabla B)\|_{L^2}^2 + C \delta^2 \|B\|_{H^2}^2.
\end{aligned}
\end{equation*}
By Young's inequality,
\begin{equation*}
\begin{aligned}
M_3\leq \frac{\chi}{2} \|\nabla u\|_{\dot{H}^{k+1}}^2 + \frac{\chi}{2} \|\nabla^{k} (\alpha\cdot\nabla B)\|_{L^2}^2.
\end{aligned}
\end{equation*}
\begin{equation*}
\begin{aligned}
M_4\leq \frac{2-\chi}{16} \|\nabla^{k} (\alpha\cdot\nabla B)\|_{L^2}^2 + C \|\nabla\omega\|_{\dot{H}^{k}}^2.
\end{aligned}
\end{equation*}
Substituting the above bounds into \eqref{T8},  one arrives at
\begin{equation}\label{T9}
\begin{aligned}
&\frac{2-\chi}{4} \|\nabla^{k} (\alpha\cdot\nabla B)\|_{L^2}^2 - \frac{d}{dt}\int \nabla^{k} u \cdot \nabla^{k} (\alpha\cdot\nabla B)\\
&\leq (\frac{\chi}{2} + C \delta^2) \|\nabla u\|_{\dot{H}^{k+1}}^2 + |\alpha|^2 \|\nabla u\|_{\dot{H}^{k}}^2 + C \|\nabla\omega\|_{\dot{H}^{k}}^2 + C \delta^2 \|B\|_{H^2}^2.
\end{aligned}
\end{equation}
As in the estimate of \eqref{T9},
\begin{equation}\label{T10}
\begin{aligned}
&\frac{2-\chi}{4} \|\alpha\cdot\nabla B\|_{L^2}^2 - \frac{d}{dt}\int u \cdot (\alpha\cdot\nabla B)\\
&\leq C \delta \|\nabla u\|_{\dot{H}^{1}}^2 + |\alpha|^2 \|\nabla u\|_{L^2}^2 + C \|\nabla\omega\|_{L^2}^2 + C \delta^2 \|B\|_{H^2}^2.
\end{aligned}
\end{equation}
Adding \eqref{T9} and \eqref{T10} up, together with Lemma \ref{DC}, we obtain
\begin{equation}\label{T11}
\begin{aligned}
&\frac{2-\chi}{4} \|\alpha\cdot\nabla B\|_{H^{r+3}}^2 - \frac{d}{dt}\sum_{k=0}^{r+3}\int \nabla^{k} u \cdot \nabla^{k} (\alpha\cdot\nabla B)\\
&\leq (\frac{\chi}{2} + C \delta^2) \|\nabla u\|_{H^{r+4}}^2  + |\alpha|^2 \|\nabla u\|_{H^{r+3}}^2 + C \|\nabla\omega\|_{H^{r+3}}^2 + C \delta^2 \|B\|_{H^{2}}^2\\
&\leq (\frac{\chi}{2} + C \delta^2 + |\alpha|^2) \|\nabla u\|_{H^{r+4}}^2 + C \|\nabla\omega\|_{H^{r+5}}^2 + C \delta^2 \|\alpha\cdot\nabla B\|_{H^{r+3}}^2.
\end{aligned}
\end{equation}

By choosing $\delta$ sufficiently small, and noting that $|\alpha|^2<\chi<2$, then \eqref{T7} and \eqref{T11} imply that there exists $C_0>0$ such that
\begin{equation}\label{T12}
\begin{aligned}
&C_0 (\|u\|_{H^{r+5}}^2 + \|\alpha\cdot\nabla B\|_{H^{r+3}}^2) - \frac{d}{dt}\sum_{k=0}^{r+4}\int \nabla^{k} \omega \cdot \nabla^{k} (\nabla\times u)\\
&\ \ \  - \frac{d}{dt}\sum_{k=0}^{r+3}\int \nabla^{k} u \cdot \nabla^{k} (\alpha\cdot\nabla B)\\
&\leq C_1 \|\nabla\omega\|_{H^{r+5}}^2.
\end{aligned}
\end{equation}

\vskip .1in
\textbf{Step III. Completion of the bootstrapping argument.}

First of all, it follows from \eqref{Hs}, \eqref{T12} and \eqref{g1} that
\begin{equation}\label{Hs2}
\begin{aligned}
&\frac{d}{dt} \Big(\gamma \|(u, \omega, B)\|_{H^{r+5}}^2 - \sum_{k=0}^{r+4}\int \nabla^{k} \omega \cdot \nabla^{k} (\nabla\times u)  - \sum_{k=0}^{r+3}\int \nabla^{k} u \cdot \nabla^{k} (\alpha\cdot\nabla B) \Big)\\
&\ \ \ + (2\eta \gamma - C_1) \|\nabla \omega\|_{H^{r+5}}^2 + C_0 (\|u\|_{H^{r+5}}^2 + \|\alpha\cdot\nabla B\|_{H^{r+3}}^2)\\
&\leq C \gamma \|(u, \omega, B)\|_{H^{3}} \|(u, \omega, B)\|_{H^{r+5}}^2\\
&\leq C \gamma \delta \|(u, \omega)\|_{H^{r+5}}^2 +  C \gamma \|(u, \omega, B)\|_{H^{3}} \|B\|_{H^{r+5}}^2,
\end{aligned}
\end{equation}
where $\gamma>1$ is to be determined later.

Using Poincar\'{e}'s inequality and Lemma \ref{DC},
$$
\|u\|_{H^{3}}\leq C \|u\|_{H^{r+5}},\ \|\omega\|_{H^{3}}\leq C \|\nabla \omega\|_{H^{r+5}}, \ \|B\|_{H^{3}}\leq C \|\alpha\cdot\nabla B\|_{H^{r+3}}.
$$
Applying the Gagliardo-Nirenberg inequality, together with Lemma \ref{DC} and \eqref{g1}, we obtain
$$
\|B\|_{H^{r+5}}\leq C \|B\|_{H^{3}}^{\frac12} \|B\|_{H^{N}}^{\frac12} \leq C \delta^{\frac12} \|\alpha\cdot\nabla B\|_{H^{r+3}}^{\frac12}
$$
as long as $N\geq 2r+7$.

Therefore,
$$
\|(u, \omega, B)\|_{H^{3}} \|B\|_{H^{r+5}}^2 \leq C \delta(\|u\|_{H^{r+5}}^2 + \|\alpha\cdot\nabla B\|_{H^{r+3}}^2 + \|\nabla \omega\|_{H^{r+5}}^2).
$$
Then by taking $\gamma>\max\Big\{1, \frac{C_1}{\eta}\Big\}$ and $\delta<\min\Big\{\frac{\eta}{C\gamma}, \frac{C_0}{2C\gamma}\Big\}$ in \eqref{Hs2}, we get
\begin{equation}\label{Hs3}
\begin{aligned}
&\frac{d}{dt} E(t) + D(t)\leq 0,
\end{aligned}
\end{equation}
where
\begin{equation*}
\begin{aligned}
&E(t)=\gamma \|(u, \omega, B)(t)\|_{H^{r+5}}^2 - \sum_{k=0}^{r+4}\int \nabla^{k} \omega \cdot \nabla^{k} (\nabla\times u)  - \sum_{k=0}^{r+3}\int \nabla^{k} u \cdot \nabla^{k} (\alpha\cdot\nabla B),\\
&D(t)=(\gamma-1) \eta \|\nabla \omega(t)\|_{H^{r+5}}^2 + \frac{C_0}{2} (\|u(t)\|_{H^{r+5}}^2 + \|(\alpha\cdot\nabla B)(t)\|_{H^{r+3}}^2).
\end{aligned}
\end{equation*}
Clearly, we can take $\gamma>1$ such that
$$
E(t)\geq \|(u, \omega, B)(t)\|_{H^{r+5}}^2.
$$
In addition, if we take $N\geq 4r+11$, we obtain by the Gagliardo-Nirenberg inequality and Lemma \ref{DC} that
$$
\|B\|_{H^{r+5}}\leq C \|B\|_{H^{3}}^{\frac34} \|B\|_{H^{N}}^{\frac14} \leq C  \|\alpha\cdot\nabla B\|_{H^{r+3}}^{\frac34} \|B\|_{H^{N}}^{\frac14}
$$
which together with \eqref{g1} and the Poincar\'{e} inequality gives
$$
E(t)\leq C \delta^{\frac12} D(t)^{\frac34}.
$$
Then it follows from this and \eqref{Hs3} that
\begin{equation*}
\begin{aligned}
&\frac{d}{dt} E(t) + C_2 E(t)^{\frac43}\leq 0.
\end{aligned}
\end{equation*}
This immediately yields
\begin{equation}\label{DE}
\begin{aligned}
E(t)\leq C(1+t)^{-3}.
\end{aligned}
\end{equation}

As in the estimate of \eqref{Hs}, we have
\begin{equation*}
\begin{aligned}
&\frac{d}{dt} \|(u, \omega, B)\|_{H^{N}}^2 + 2\eta \|\nabla \omega\|_{H^{N}}^2\\
&\leq C \|(u, \omega, B)\|_{H^{3}} \|(u, \omega, B)\|_{H^{N}}^2.
\end{aligned}
\end{equation*}
Using Gronwall's inequality and \eqref{DE}, we derive that
\begin{equation*}
\begin{aligned}
\|(u, \omega, B)(t)\|_{H^{N}}^2
&\leq \|(u_0, \omega_0, B_0)\|_{H^{N}}^2 e^{C\int_0^t \|(u, \omega, B)(\tau)\|_{H^{3}} d\tau}\\
&\leq C\epsilon^{2} e^{C\int_0^t (1+\tau)^{-\frac32} d\tau}\\
&\leq C\epsilon^{2}.
\end{aligned}
\end{equation*}
By taking $\epsilon$ small enough such that $C \epsilon^{2}<\frac{\delta}{2}$, then the above inequality implies \eqref{g2} for all $t\in[0, T]$. Then the bootstrapping argument implies that $T=\infty$ and \eqref{g2} holds for all $t<\infty$. Thus the proof of Theorem \ref{main2} is completed.

\end{proof}

\vskip .3in
\section*{Acknowledgements}
\vskip .1in

This work was supported by National Natural Science Foundation of China (No. 12371232).





\vskip .2in
\begin{thebibliography}{100}


\bibitem{AZ17} H. Abidi, P. Zhang, On the global solution of a 3-D MHD system with initial data near equilibrium, Comm. Pure Appl. Math. 70 (2017), 1509-1561.

\bibitem{AS74} G. Ahmadi, M. Shahinpoor, Universal stability of magneto-micropolar fuid motions, Int. J. Eng. Sci. 12 (1974), 657-663.

\bibitem{BSS88} C. Bardos, C. Sulem, P.-L. Sulem, Longtime dynamics of a conductive fluid in the presence of a strong magnetic field, Trans. Amer. Math. Soc. 305 (1988), 175-191.

\bibitem{BCFZ19} P. Braz e Silva, F.W. Cruz, L.B.S. Freitas, P.R. Zingano, On the $L^2$ decay of weak solutions for the 3D asymmetric fluids equations, J. Differ. Equ. 267 (2019), 3578-3609.

\bibitem{CL18} Y. Cai, Z. Lei, Global well-posedness of the incompressible magnetohydrodynamics, Arch. Ration. Mech. Anal. 3 (2018), 969-993.

\bibitem{CW11} C. Cao, J. Wu, Global regularity for the 2D MHD equations with mixed partial dissipation and magnetic diffusion, Adv. Math. 226 (2011), 1803-1822.

\bibitem{CM12} Q. Chen, C. Miao, Global well-posedness for the micropolar fluid system in critical Besov spaces, J. Differ. Equ. 252 (2012), 2698-2724.

\bibitem{CZZ22} W. Chen, Z. Zhang, and J. Zhou, Global well-posedness for the 3-D MHD equations with partial diffusion in the periodic domain, Sci. China Math. 65 (2022), 309-318.

\bibitem{DLW17} B. Dong, J. Li, J. Wu, Global well-posedness and large-time decay for the 2D micropolar equations, J. Differ. Equ. 262 (2017), 3488-3523.

\bibitem{DZ10} B. Dong, Z. Zhang, Global regularity of the 2D micropolar fluid flows with zero angular viscosity, J. Differ. Equ. 249 (2010), 200-213.

\bibitem{Eringen66} A. Eringen, Theory of micropolar fluids, J. Math. Mech. 16 (1966), 1-18.

\bibitem{Ferrari83} C. Ferrari, R. Gilbert, On lubrication with structured fluids, Appl. Anal. 15 (1983), 127-146.

\bibitem{GMNPZ25} R.H. Guterres, W.G. Melo, C.J. Niche, C.F. Perusato, P.R. Zingano, Strong alignment of micro-rotation and vorticity in 3D micropolar flows, Nonlinearity 38 (2025), No. 015006.

\bibitem{HXY18} L. He, L. Xu, P. Yu, On global dynamics of three dimensional magnetohydrodynamics: nonlinear stability of Alfv\'{e}n waves, Ann. PDE 4 (2018), No. 5.

\bibitem{KPV91} C.E. Kenig, G. Ponce, L. Vega, Well-posedness of the initial value problem for the Korteweg-de Vries equation,  J.  Amer. Math. Soc. 4 (1991), 323-347.

\bibitem{LZ09} Z. Lei, Y. Zhou, BKM's criterion and global weak solutions for magnetohydrodynamics with zero viscosity, Discrete Contin. Dyn. Syst. 25 (2009), 575-583.

\bibitem{LXZ15} F. Lin, L. Xu, P. Zhang, Global small solutions of 2-D incompressible MHD system, J. Differ. Equ. 259 (2015), 5440-5485.

\bibitem{LZ14} F. Lin, P. Zhang, Global small solutions to an MHD-type system: the three-dimensional case, Comm. Pure Appl. Math. 67 (2014), 531-580.

\bibitem{Lukaszewicz99} G. ${\L}$ukaszewicz, Micropolar fluids. Theory and applications. Modeling and Simulation in Science, Engineering and Technology. Birkh\"{a}user Boston, Inc., Boston, MA, 1999.

\bibitem{LS18} M. Li, H. Shang, Large time decay of solutions for the 3D magneto-micropolar equations, Nonlinear Anal. Real World Appl. 44 (2018), 479-496.

\bibitem{NS24} D. Niu, H. Shang, Lower and upper bounds of decay to the d-dimensional magneto-micropolar equations, J. Math. Phys. 65 (2024), No. 121501.

\bibitem{PZZ18} R. Pan, Y. Zhou, Y. Zhu, Global classical solutions of three dimensional viscous MHD system without magnetic diffusion on periodic boxes, Arch. Ration. Mech. Anal. 227 (2018), 637-662.

\bibitem{RWXZ14} X. Ren, J. Wu, Z. Xiang, Z. Zhang, Global existence and decay of smooth solution for the 2-D MHD equations without magnetic diffusion, J. Funct. Anal. 267 (2014), 503-541.

\bibitem{RB98} M. Rojas-Medar, J. Boldrini, Magneto-micropolar fluid motion: existence of weak solutions, Rev. Mat. Complut.  11  (1998),  443-460.

\bibitem{SD85} V. Sastry, T. Das, Stability of Couette fow and Dean fow in micropolar fluids, Int. J. Engin. Sci. 23 (1985), 1163-1177.

\bibitem{SL25} H. Shang, C. Liu, Global well-posedness and large time behavior for the 3D anisotropic micropolar equations, J. Differ. Equ. 421 (2025), 531-557.

\bibitem{SS26} H. Shang, W. Song, Stability and exponential decay of micropolar equations with zero kinematic viscosity, Appl. Math. Lett. 177 (2026), No. 109902.

\bibitem{SW21}  H. Shang, J. Wu, Global regularity for 2D fractional magneto-micropolar equations, Math. Z. 297 (2021), 775-802.

\bibitem{SZ17}  H. Shang, J. Zhao, Global regularity for 2D magneto-micropolar equations with only micro-rotational velocity dissipation and magnetic diffusion, Nonlinear Anal. 150 (2017), 194-209.

\bibitem{TWZ19} Z. Tan, W. Wu, J. Zhou, Global existence and decay estimate of solutions to magneto-micropolar fluid equations, J. Differ. Equ. 266 (2019), 4137-4169.

\bibitem{WZ17} D. Wei, Z. Zhang, Global well-posedness of the MHD equations in a homogeneous magnetic field, Anal. PDE 10 (2017), 1361-1406.

\bibitem{WZ20} D. Wei, Z. Zhang, Global well-posedness for the 2-D MHD equations with magnetic diffusion, Commun. Math. Res. 36 (2020), 377-389.

\bibitem{WZ23} W. Wu, Y. Zhang, Non-uniform decay of solutions to the incompressible magneto-micropolar fluids with/without magnetic diffusion and spin viscosity, J. Math. Phys. 64 (2023), No. 111502.

\bibitem{XJL24} Y. Xie, Q. Jiu, J. Liu, Sharp decay estimates and asymptotic stability for incompressible MHD equations without viscosity or magnetic diffusion, Calc. Var. Partial Differ. Equ. 63 (2024), No. 191.

\bibitem{Xue11} L. Xue, Wellposedness and zero microrotation viscosity limit of the 2D micropolar fluid equations, Math. Methods Appl. Sci. 34 (2011), 1760-1777.

\bibitem{Yamazaki15} K. Yamazaki, Global regularity of the two-dimensional magneto-micropolar fluid system with zero angular viscosity, Discrete Contin. Dyn. Syst. 35 (2015), 2193-2207.

\bibitem{YY22} W. Ye, Z. Yin, Global well-posedness for the non-viscous MHD equations with magnetic diffusion in critical Besov spaces, Acta Math. Sin. (Engl. Ser.) 38 (2022), 1493-1511.

\bibitem{YQ18} B. Yuan, Y. Qiao, Global regularity for the 2D magneto-micropolar equations with partial and fractional dissipation, Comput. Math. Appl. 76 (2018), 2345-2359.

\bibitem{Zhai23} X. Zhai, Stability for the 2D incompressible MHD equations with only magnetic diffusion, J. Differ. Equ. 374 (2023), 267-278.

\bibitem{Zhang16} T. Zhang, Global solutions to the 2D viscous, non-resistive MHD system with large background magnetic field, J. Differ. Equ. 260 (2016), 5450-5480.

\bibitem{ZZ23} Y. Zheng, Y. Zhu, Stability of 2D inviscid MHD equations with only vertical magnetic diffusion on $\mathbb T^2$, J. Math. Phys. 64 (2023), No. 111508.

\bibitem{ZZ18} Y. Zhou, Y. Zhu, Global classical solutions of 2D MHD system with only magnetic diffusion on periodic domain, J. Math. Phys. 59 (2018), No. 81505.





\end{thebibliography}
\end{document}